\documentclass[10pt, reqno]{amsart}
\usepackage{amssymb}
\usepackage{hyperref}
\usepackage{cite}

\numberwithin{equation}{section}

\newtheorem{theorem}{Theorem}[section]

\newtheorem{lemma}[theorem]{Lemma}

\newtheorem{corollary}[theorem]{Corollary}

\newtheorem{proposition}[theorem]{Proposition}

\theoremstyle{definition}
\newtheorem{definition}[theorem]{Definition}

\newtheorem{remarks}[theorem]{Remarks}

\theoremstyle{remark}

\newcommand{\R}{\mathbb{R}}
\newcommand{\T}{\mathbb{T}}
\newcommand{\C}{\mathbb{C}}

\newcommand{\E}{\mathbb{E}}

\newcommand{\Z}{\mathbb{Z}}
\newcommand{\N}{\mathcal{N}}

\newcommand{\Prad}{P_{\text{rad}}}
\newcommand{\Imax}{I_{\text{max}}}
\newcommand{\x}{\langle x\rangle}

\renewcommand{\P}{\mathbb{P}}
\newcommand{\EE}{\mathbf{E}}

\let\Re=\undefined\DeclareMathOperator*{\Re}{Re}
\let\Im=\undefined\DeclareMathOperator*{\Im}{Im}

\newcommand{\eps}{\varepsilon}

\DeclareMathOperator{\rad}{rad}

\newcommand{\jb}{\langle\nabla\rangle}
\newcommand{\qtq}[1]{\quad\text{#1}\quad}

\begin{document}

\title[Almost sure scattering]{Almost sure scattering for the energy-critical NLS with radial data below $H^1(\R^4)$}

\author[R. Killip]{Rowan Killip}
\address{Department of Mathematics, UCLA, Los Angeles, USA}
\email{killip@math.ucla.edu}

\author[J. Murphy]{Jason Murphy}
\address{Department of Mathematics and Statistics, Missouri University of Science and Technology, Rolla, USA}
\email{jason.murphy@mst.edu}

\author[M. Visan]{Monica Visan}
\address{Department of Mathematics, UCLA, Los Angeles, USA}
\email{visan@math.ucla.edu}

\begin{abstract}
We prove almost sure global existence and scattering for the energy-critical nonlinear Schr\"odinger equation with randomized spherically symmetric initial data in $H^s(\R^4)$ with $\frac56<s<1$.  We were inspired to consider this problem by the recent work of Dodson--L\"uhrmann--Mendelson \cite{DLM}, which treated the analogous problem for the energy-critical wave equation.
\end{abstract}

\maketitle

\section{Introduction}\label{S:intro}

We consider the initial-value problem for the defocusing cubic nonlinear Schr\"odinger equation (NLS) in four space dimensions:
\begin{equation}\label{nls}
(i\partial_t + \Delta) u = |u|^2 u.
\end{equation}
This equation is {energy-critical} in four dimensions:  the rescaling that preserves the class of solutions to \eqref{nls}, namely,
\[
u(t,x)\mapsto \lambda u(\lambda^2 t,\lambda x),
\]
also leaves invariant the conserved {energy}, defined by
\begin{equation}\label{def:E}
E[u(t)] = \int_{\R^4}\tfrac12|\nabla u(t,x)|^2 + \tfrac14|u(t,x)|^4\,dx.
\end{equation}

Equation \eqref{nls} is known to be globally well-posed in the energy space.  More precisely, we have the following:

\begin{theorem}[Well-posedness in the energy space; \cite{RV, Visan4D}]\label{T:EC}\leavevmode\ %
Let $u_0\in \dot H^1(\R^4)$.  Then there exists a unique global solution $u\in C_t\dot H^1_x(\R\times\R^4)$ to \eqref{nls} with $u(0)=u_0$.  Moreover, the solution satisfies
$$
\|u\|_{L_t^4 L_x^8(\R\times\R^4)} \leq L(E(u_0)).
$$
Consequently, there exist scattering states $u_\pm\in \dot H^1(\R^4)$ such that
$$
\|u(t)-e^{it\Delta}u_\pm\|_{\dot H^1_x}\to 0\qtq{as} t\to \pm\infty.
$$
\end{theorem}

On the other hand, Christ--Colliander--Tao \cite{CCT} showed that the data-to-solution map for \eqref{nls} is discontinuous at the origin in the $H^s(\R^4)$ topology whenever $s<1$.  In this paper, we prove that suitably randomized spherically symmetric initial data in $H^{s}(\R^4)$ with $\frac56<s<1$ lead to global scattering solutions almost surely.

\begin{definition}[Randomization]\label{def:random} Let $\varphi$ be a bump function supported in the unit ball such that
\[
\sum_{k\in\Z^4} \varphi_k(\xi)= 1\qtq{for all $\xi\in \R^4$, where} \varphi_k(\xi) := \varphi(\xi-k).
\]
Fix $s\in \R$ and $f\in H^s(\R^4)$.  For $k\in\Z^4$, we define
$$
f_k:= [\hat f \varphi_k]^{\vee} = f* \check{\varphi}_k.
$$
Let $\{X_k\}_{k\in\Z^4}$ be independent, mean zero, real or complex Gaussian random variables of uniformly bounded variance.  We will write the underlying probability space as $(\Omega,\Sigma,\P)$.  We define the randomization of $f$ via
\[
f^\omega(x) = \sum_{k\in\Z^4} X_k f_k(x).
\]
\end{definition}

For concreteness, in this paper we work with the Gaussian randomization introduced above.  Use of Kinchine's inequality would allow one to treat more general randomizations, such as those satisfying
$$
\E(e^{\gamma X_k})\leq e^{c\gamma^2}
$$
uniformly for $\gamma\in \R$, $k\in \Z^4$, for some $c>0$.

\begin{remarks}\label{Remarks}
(i) Note that $m(i\nabla)f^\omega=[m(i\nabla)f]^\omega$ for any Fourier multiplier operator $m$.  We also have $(f+g)^\omega = f^\omega + g^\omega$. \\[2mm]
(ii) For any $s\in \R$ we have $\E(\|f^\omega\|_{H^s}^2)\sim \|f\|_{H^s}^2$.\\[2mm]
(iii) Even if the function $f$ is radial, the randomization $f^\omega$ is not.
\end{remarks}

Our main result is the following:

\begin{theorem}[Almost sure scattering]\label{T} Fix $\frac56<s<1$ and a spherically symmetric function $f\in H^s(\R^4)$.  For almost every $\omega$, there exists a unique global solution $u$ to \eqref{nls} with $u(0)=f^\omega$.  Furthermore, $u$ scatters in the following sense: there exist unique $u_\pm\in H^1(\R^4)$ such that
\[
\lim_{t\to\pm\infty} \|u(t) - e^{it\Delta}[f^\omega + u_\pm]\|_{H^1(\R^4)} = 0.
\]
\end{theorem}

Uniqueness in Theorem~\ref{T} holds in the following sense: Writing $u(t) = e^{it\Delta}f^\omega + v(t)$, there exists a unique global solution $v\in C_tH^1_x(\R\times\R^4)\cap L_t^4L_x^8 (\R\times\R^4)$ to
\begin{align*}
(i\partial_t+\Delta) v = |v+e^{it\Delta}f^\omega|^2(v+e^{it\Delta}f^\omega) \qtq{with} v(0)=0.
\end{align*}

A substantial body of work on dispersive equations with randomized initial data has built up over the last two decades.  Correspondingly, we must curtail our presentation here and primarily discuss works concerned with the energy-critical wave and Schr\"odinger problems on Euclidean space.  See also \cite{NS} for a proof of almost sure well-posedness for the energy-critical NLS on $\T^3$. 

Almost sure global well-posedness for supercritical data, randomized as in Definition~\ref{def:random}, was proved by Pocovnicu \cite{Pocovnicu} and Oh--Pocovnicu \cite{Oh-Pocovnicu} for the energy-critical wave equation and by Benyi--Oh--Pocovnicu \cite{Benyi-Oh-Pocovnicu} and Brereton \cite{Brereton} for the energy-critical Schr\"odinger equation.  These works also establish scattering with positive probability for small randomized data.  We should note that the results in \cite{Benyi-Oh-Pocovnicu, Brereton} are conditional on energy-critical bounds satisfied by the function $v$ introduced above.  In \cite{DLM}, Dodson--L\"uhrmann--Mendelson proved almost sure scattering for the four-dimensional energy-critical wave equation with (large) supercritical radial data, randomized as in Definition~\ref{def:random}.  In this paper we establish the analogous result for the energy-critical Schr\"odinger equation, Theorem~\ref{T}; in particular, our global well-posedness result is not conditional on bounds satisfied by $v$.

Many prior works considered energy-critical and -subcritical problems on Euclidean space, mostly with different randomizations; see, for example, \cite{BTT, Deng, Suzzoni1, Suzzoni2, LM1, LM2, Murphy, Poiret1, Poiret2, PRT, Thomann}.  We wish to draw particular attention to \cite[Theorem~1.3]{PRT}, which establishes scattering for the energy-critical Schr\"odinger equation with positive probability for a particular ensemble of random initial data which is merely $L^2_x$.

The proof of Theorem~\ref{T} relies on the further development of the methods introduced in the papers described above, particularly \cite{DLM, Oh-Pocovnicu, Pocovnicu}. The first step is to  regard the equation satisfied by $v$ as a perturbation of the energy-critical problem.  Specifically, we write
\begin{align}\label{v eqn}
(i\partial_t+\Delta) v = |v|^2v + \bigl[|v+e^{it\Delta}f^\omega|^2(v+e^{it\Delta}f^\omega)-|v|^2v\bigr].
\end{align}
The fact that the stability theory for the energy-critical NLS is the right tool to study energy-critical equations with perturbations was first observed by X. Zhang in \cite{Zhang} and elaborated on in \cite{Matador}.  The utility of this approach in the energy-critical random-data setting was first observed by O. Pocovnicu in \cite{Pocovnicu}.

Relying on Theorem~\ref{T:EC}, we develop a stability theory (along pre-existing lines) tailored to equation \eqref{v eqn}.  This allows us to show that there exists a unique global solution $v$ to \eqref{v eqn} that scatters in $H^1_x$, provided we can verify two conditions: (1) $v$ satisfies uniform energy bounds on its lifespan and (2) the error $|v+e^{it\Delta}f^\omega|^2(v+e^{it\Delta}f^\omega)-|v|^2v$ is controlled in suitable scaling-critical spaces.  As we will see in Section~\ref{S:wp}, the second condition above is satisfied as long as the forcing term $e^{it\Delta} f^\omega$ obeys certain spacetime bounds.  Thus, building on the stability result we develop for \eqref{v eqn} (see Lemma~\ref{L:stab2}), we show in Proposition~\ref{P:v-scatter} that the proof of Theorem~\ref{T} reduces to demonstrating uniform energy bounds for $v$ on its lifespan and certain spacetime bounds for the free evolution of the randomized data.

In Section~\ref{S:E}, we show that if the forcing term $e^{it\Delta} f^\omega$ obeys some further spacetime bounds (see \eqref{X} and \eqref{3:30}), then $v$ is uniformly bounded in $H^1_x$ on its lifespan.  To achieve this, we run a double bootstrap argument relying on an estimate on the energy increment of $v$ (see Lemma~\ref{P:energy}) and a Morawetz-type inequality  (see Lemma~\ref{P:Morawetz}).  Instead of the standard Lin--Strauss Morawetz weight $a(x)=|x|$, we prove an estimate based on the weight $a(x)=\langle x\rangle$.  The additional convexity of this weight gains us much-needed time integrability for $\nabla v$, albeit in weighted spaces.

In Section~\ref{S:notation}, we prove that for spherically symmetric $f\in H^{s}(\R^4)$ with $s>\frac56$, the random free evolution $e^{it\Delta} f^\omega$ almost surely obeys the spacetime bounds needed to run all the arguments described above (see Proposition~\ref{P:STE2} and Proposition~\ref{P:STE}).  The key ingredients here are weighted radial Strichartz estimates (see Proposition~\ref{P:RS}) and the local smoothing estimate (see Lemma~\ref{L:LS}), combined with the moment bounds in Lemma~\ref{L:SF}.

\subsection*{Acknowledgements}  R. K. was supported by NSF grant DMS-1600942.  J. M. was supported in part by NSF DMS-1400706.  M. V. was supported by NSF grant DMS-1500707.

\section{Notation and useful lemmas}\label{S:notation}
We write $A\lesssim B$ to indicate that $A\leq CB$ for some constant $C>0$. Dependence of implicit constants on various parameters will be indicated with subscripts. For example $A\lesssim_\varphi B$ means that $A\leq CB$ for some $C=C(\varphi)$.  Implicit constants will always be permitted to depend on the parameters in the randomization.  We write $A\sim B$ if $A\lesssim B$ and $B\lesssim A$. We write $A\ll B$ if $A\leq cB$ for some small $c>0$.

We write $L_x^r$, $H_x^s$, and $W_x^{s,r}$ for the usual Lebesgue and Sobolev spaces.  We also use mixed space-time norms, e.g. $L_t^q L_x^r$ and $L_t^q W_x^{s,r}$.  We write $H^s_{\rad}$ to denote the space of spherically symmetric functions in $H_x^s$. 

We use the standard Littlewood--Paley projection operators $P_N$ with the understanding that $P_1$ denotes the operator $P_{\leq 1}$.  Summation in $N$ will always be taken over $N\in 2^{\mathbb{N}}=\{1,2,4, \ldots\}$.  The Littlewood--Paley operators obey the following well-known Bernstein estimates:

\begin{lemma}[Bernstein estimates] For $1\leq r\leq q\leq\infty$ and $s\geq 0$ we have
\begin{align*}
\|\vert\nabla\vert^s P_Nu\|_{L_x^r(\R^d)}&\lesssim N^s \|P_N u\|_{L_x^r(\R^d)}\\
\|P_Nu\|_{L_x^q(\R^d)}&\lesssim N^{\frac dr-\frac dq} \|P_Nu\|_{L_x^r(\R^d)}.
\end{align*} 	
\end{lemma}

Next, we record two simple weighted estimates.

\begin{lemma}\label{L:WY} For $1\leq r\leq m\leq\infty$, $\beta>0$, and $\phi\in \mathcal S(\R^d)$,
\[
\|\x^\beta[|\phi|\ast |u|]\|_{L_x^m(\R^d)} \lesssim \|\x^\beta u\|_{L_x^r(\R^d)}.
\]
\end{lemma}

\begin{proof} Using the rapid decay of $\phi$, the triangle inequality, H\"older's inequality, and Minkowski's integral inequality, we estimate for any $A>0$,
\begin{align*}
\|\x^\beta[|\phi|\ast|u|]\|_{L_x^m} & \lesssim \biggl\| \int_{|x-y|\leq 1}\langle y\rangle^\beta |u(y)|\,dy\biggr\|_{L_x^m} +
\sum_{R\geq1} \biggl\|\int_{|x-y|\sim R}\frac{\x^\beta |u(y)|}{\langle x-y\rangle^{A}} dy\biggr\|_{L_x^m} \\
& \lesssim \|\langle y\rangle^\beta\chi_{|x-y|\leq 1}u\|_{L_x^m L_y^r} + \sum_{R\geq 1} R^{-A+\beta+\frac{d}{r'}}\|\langle y\rangle^\beta \chi_{|x-y|\sim R}u\|_{L_x^m L_y^r}  \\
&\lesssim \|\langle y\rangle^\beta\chi_{|x-y|\leq 1}u\|_{L_y^r L_x^m} + \sum_{R\geq 1} R^{-A+\beta+\frac{d}{r'}}\|\langle y\rangle^\beta \chi_{|x-y|\sim R}u\|_{L_y^r L_x^m} \\
& \lesssim \|\langle y\rangle^\beta u\|_{L_y^r} + \sum_{R\geq 1} R^{-A+\beta+\frac{d}{r'}+\frac{d}{m}}\|\langle y\rangle^\beta u\|_{L_y^r}.
\end{align*}
For $A$ large enough, we can sum over $R\in 2^{\mathbb{N}}$ to complete the proof.
\end{proof}

\begin{lemma}\label{L:commutator} For $0\leq\beta\leq 1$ and $d<m<\infty$,
\[
\|\x^\beta u\|_{L_x^\infty(\R^d)}\lesssim \sum_{N\geq 1} N^{\frac{d}{m}}\|\x^\beta P_N u\|_{L_x^m(\R^d)}.
\]
\end{lemma}

\begin{proof} To begin, we apply Bernstein to estimate
\begin{align*}
\|\x^\beta u\|_{L_x^\infty} &\lesssim \sum_{N\geq 1}\|P_N[\x^{\beta}u]\|_{L_x^\infty}  \lesssim \sum_{N\geq 1} N^{\frac{d}{m}}\|P_N[\x^\beta u]\|_{L_x^m} \\
& \quad\lesssim \sum_{N\geq 1} N^{\frac{d}{m}}\|\x^\beta P_N u\|_{L_x^m} + \sum_{N\geq 1}N^{\frac{d}{m}}\|[\x^\beta,P_N]u\|_{L_x^m}.
\end{align*}

Writing $\phi(\cdot/N)$ for the multiplier of $P_N$, a direct computation gives
\[
[a,P_N](x,y) = N^d \check\phi(N(x-y))[a(x)-a(y)]
\]
for any function $a$.  Thus, by Schur's test,
\[
\|[a,P_N]\|_{L^m_x\to L^m_x} \lesssim N^{-1}\|\partial a\|_{L^\infty_x}.
\]
Applying this with $a(x)=\x^\beta$ for $0\leq\beta\leq 1$, we find
\[
\sum_{N\geq 1}N^{\frac{d}{m}}\|[\x^\beta,P_N]u\|_{L_x^m} \lesssim \|u\|_{L_x^m}\lesssim \sum_{N\geq 1} N^{\frac{d}{m}}\|\x^\beta P_N u\|_{L_x^m},
\]
where we used $d<m$ to derive the first inequality above. \end{proof}

\subsection{The linear Schr\"odinger equation}
The standard dispersive estimate for the linear propagator $e^{it\Delta}$ in four space dimensions follows from the kernel estimate $|e^{it\Delta}(x,y)| \lesssim |t|^{-2}.$  This bound, together with the unitarity of the linear propagator in $L^2_x$, implies the full range of Strichartz estimates:

\begin{proposition}[Strichartz estimates]\label{P:strichartz} Let $2\leq q_1,q_2\leq\infty$ and $r_j=\frac{2q_j}{q_j-1}$.  Let $I$ be a time interval with $t_0\in\bar I$.  Then
\begin{align*}
\| e^{it\Delta}f\|_{L_t^{q_1}L_x^{r_1}(\R\times\R^4)}&\lesssim \|f\|_{L_x^2(\R^4)}, \\
\biggl\| \int_{t_0}^t e^{i(t-s)\Delta}F(s)\,ds\biggr\|_{L_t^{q_1}L_x^{r_1}(I\times\R^4)} &\lesssim \|F\|_{L_t^{q_2'}L_x^{r_2'}(I\times\R^4)}.
\end{align*}
\end{proposition}

For radial functions, one has additional estimates.  Letting $\Prad$ denote the projection onto radial functions, one has the following kernel estimate from \cite{KVZ}:
\begin{equation}\label{rde}
|e^{it\Delta}\Prad(x,y)| \lesssim |t|^{-\frac12}|x|^{-\frac32}|y|^{-\frac32}.
\end{equation}
Combined with the standard dispersive estimate, this leads to
\[
|e^{it\Delta}\Prad(x,y)| \lesssim |t|^{-\frac2q}|x|^{-\frac{2(q-1)}{q}}|y|^{-\frac{2(q-1)}{q}}\qtq{for all} 1\leq q\leq 4.
\]
Combining this with the standard $TT^*$ argument leads to the following weighted radial Strichartz estimates:
\begin{proposition}[Weighted radial Strichartz]\label{P:RS} For $f\in L^2_{\rad}(\R^4)$ and $2<q\leq 4$,
\[
\| |x|^{\frac{2(q-1)}{q}} e^{it\Delta} f\|_{L_t^q L_x^\infty(\R\times\R^4)} \lesssim \|f\|_{L_x^2(\R^4)}.
\]
\end{proposition}

Interpolating the estimates of Propositions~\ref{P:strichartz} and \ref{P:RS} yields the following:

\begin{corollary}\label{C:RS} For $f\in L^2_{\rad}(\R^4)$, $2<q\leq 4$, and $0\leq \beta\leq \frac{2(q-1)}{q}$,
\[
\||x|^\beta e^{it\Delta} f\|_{L_t^q L_x^{\frac{4q}{2(q-1)-\beta q}}(\R\times\R^4)}\lesssim \|f\|_{L_x^2(\R^4)}.
\]
\end{corollary}

We will rely on local smoothing estimates (cf. \cite{ConsSaut,Sjolin87,Vega88}) to absorb some of the derivatives landing on the randomized linear evolution.

\begin{lemma}[Local smoothing]\label{L:LS} For any $\eps>0$,
\[
\|\x^{-\frac12-\eps} e^{it\Delta} f\|_{L_{t,x}^2(\R\times\R^d)} \lesssim \|f\|_{\dot H_x^{-\frac12}(\R^d)}.
\]
\end{lemma}

\subsection{Almost sure bounds}

In this subsection we develop a collection of almost sure estimates on the randomized free evolution.  We start by estimating the moments of the randomized free evolution.

\begin{lemma}[Moment bounds]\label{L:SF} Let $f^\omega$ be the randomization of $f$ as in Definition~\ref{def:random}.  For $m\geq 1$,
\[
\E\bigl(|f^\omega|^{2m}\bigr)\lesssim_m \bigl(|\check\varphi|\ast|f|^2\bigr)^m.
\]
\end{lemma}

\begin{proof} As $\sum X_k f_k(x)$ is Gaussian, its moments can be computed exactly.  Specifically, we have
\begin{align*}
\E\bigl(\bigl|\sum_k X_k f_k(x)\bigr|^{2m}\bigr) & \sim_m \biggl(\sum_k |f_k(x)|^2\biggr)^m.
\end{align*}
Next, using the Poisson summation formula and Cauchy--Schwarz, we estimate
\begin{align*}
\sum_{k\in\Z^4} |f_k(x)|^2& \sim \sum_{k\in\Z^4} \iint \check{\varphi}_k(y)\overline{\check{\varphi}_k}(z) f(x-y)\overline{f}(x-z)\,dy\,dz \\
& \sim \iint \sum_{k\in\Z^4} e^{ik(y-z)}\check\varphi(y)\overline{\check{\varphi}}(z)f(x-y)\overline{f}(x-z)\,dy\,dz \\
& \sim \int \sum_{\ell\in2\pi\Z^4} \check\varphi(y)\overline{\check{\varphi}}(y-\ell) f(x-y)\overline{f}(x-y+\ell)\,dy \\
& \lesssim \int \sum_{\ell\in2\pi\Z^4}\bigl[|\check\varphi(y)\check{\varphi}(y-\ell)|+|\check{\varphi}(y)\check{\varphi}(y+\ell)|\bigr]|f(x-y)|^2\,dy \\
& \lesssim (|\check{\varphi}|\ast |f|^2)(x).
\end{align*}
This completes the proof. \end{proof}

Combining Lemmas~\ref{L:WY} and \ref{L:SF}, we derive almost sure bounds on weighted norms of the randomized free evolution.

\begin{lemma}\label{L:weighted} For $1\leq q,r\leq m<\infty$ and $\beta\geq 0$,
\begin{equation}\label{E:weighted1}
\E\biggl[\| \x^\beta e^{it\Delta}f^\omega\|_{L_t^qL_x^m}^{q}\biggr] \lesssim
	\|\x^\beta e^{it\Delta} f\|_{L_t^q L_x^r}^{q}.
\end{equation}
In particular, for $p>2$, $1\leq r_1,r_2,\frac{2p}{p-2}\leq m<\infty$, and $\beta\geq 0$,
\begin{equation}\label{E:weighted2}
\E\biggl[\| \x^\beta e^{it\Delta} f^\omega\|_{L_t^{\frac{2p}{p-2}}L_x^m}^{\frac{2p}{p-2}}\biggr] \lesssim
	\|e^{it\Delta} f\|_{L_t^{\frac{2p}{p-2}} L_x^{r_1}(\R\times B)}^{\frac{2p}{p-2}}+\||x|^\beta e^{it\Delta} f\|_{L_t^{\frac{2p}{p-2}} L_x^{r_2}(\R\times B^c)}^{\frac{2p}{p-2}},
\end{equation}
where $B$ denotes the unit ball and $B^c$ its complement.  Unless otherwise indicated, all space-time norms are over $\R\times\R^4$.
\end{lemma}

\begin{proof} Using H\"older's inequality and the assumption $m\geq q$,
\[
\text{LHS}\eqref{E:weighted1} = \int \E\bigl(\|\x^\beta e^{it\Delta} f^\omega \|_{L_x^m}^{q}\bigr)\,dt
\lesssim \int\bigl\{ \E\bigl(\|\x^\beta e^{it\Delta} f^\omega\|_{L_x^m}^m\bigr)\bigr\}^{\frac{q}{m}}\,dt.
\]
Next, by Lemmas~\ref{L:SF} and \ref{L:WY},
\begin{align*}
\int \E\bigl( \x^{\beta m}|e^{it\Delta} f^\omega |^m\bigr) \,dx\lesssim \int \x^{\beta m}[|\check\varphi|\ast|e^{it\Delta} f|^2]^{\frac{m}{2}}(x)\,dx \lesssim \|\x^\beta e^{it\Delta} f\|_{L_x^r}^m.
\end{align*}
This proves \eqref{E:weighted1}.  To derive \eqref{E:weighted2}, we write
\[
e^{it\Delta}f^\omega = [\chi e^{it\Delta} f]^\omega + [(1-\chi)e^{it\Delta}f]^\omega,
\]
where $\chi$ is the characteristic function of the unit ball, and apply the argument above to each summand.
\end{proof}

\begin{proposition}\label{P:STE2} Let $f\in L^2(\R^4)$. For $2\leq p\leq\infty$,
\[
\|e^{it\Delta} f^\omega\|_{L_t^\infty L_x^2}+ \|e^{it\Delta} f^\omega\|_{L_t^3 L_x^6}  + \|e^{it\Delta}f^\omega\|_{L_{t,x}^4} + \|e^{it\Delta} f^\omega\|_{L_t^{\frac{4p}{p+2}}L_x^4}< \infty
\]
almost surely, where all space-time norms are over $\R\times\R^4$.  If $f\in H^s(\R^4)$ for some $s>\frac12$, then we also have
\[
\|e^{it\Delta} f^\omega\|_{L_t^\infty L_x^4(\R\times\R^4)} < \infty
\]
almost surely.
\end{proposition}

\begin{proof} Almost sure finiteness of the $L_t^\infty L_x^2$ norm of the randomized free evolution follows from the unitarity of the linear propagator on $L^2_x$ and Remark~\ref{Remarks}(ii).

Using Lemma~\ref{L:weighted} (with $\beta=0$) and the Strichartz estimates, we find
\begin{align}
\E\bigl(\|e^{it\Delta}f^\omega\|_{L_t^3 L_x^6}^3\bigr) &\lesssim \|e^{it\Delta} f\|_{L_{t,x}^3}^3\lesssim \|f\|_{L_x^2}^3,  \notag\\
\E\bigl(\|e^{it\Delta} f^\omega\|_{L_{t,x}^4}^4\bigr)&\lesssim \|e^{it\Delta} f\|_{L_t^4 L_x^{\frac83}}^4\lesssim \|f\|_{L_x^2}^4,\label{4,4}\\
\E\bigl(\|e^{it\Delta}f^{\omega}\|_{L_t^{\frac{4p}{p+2}}L_x^{4}}^{\frac{4p}{p+2}}\bigr)&\lesssim \|e^{it\Delta} f\|_{L_t^{\frac{4p}{p+2}}L_x^{\frac{8p}{3p-2}}}^{\frac{4p}{p+2}}\lesssim \|f\|_{L_x^2}^{\frac{4p}{p+2}},\notag
\end{align}
where we used $p\geq 2$ for the last estimate.  Thus, these norms are finite almost surely.

Finally, we consider the $L_t^\infty L_x^4$ norm.  We begin with a general estimate:
\begin{equation}\label{E:gen}
\|F\|_{L_t^\infty L_x^4(\R\times\R^4)}^4 \lesssim \delta^{-1}\|F\|_{L_{t,x}^4(\R\times\R^4)}^4 + \delta^3\|\partial_t F\|_{L_{t,x}^4(\R\times\R^4)}^4\qtq{for any}\delta>0.
\end{equation}
To prove this, first fix a bounded interval $I\subset\R$.  By the fundamental theorem of calculus,
\[
\|F\|_{L_t^\infty L_x^4(I\times\R^4)} \leq \|F(t_0)\|_{L_x^4(\R^4)} + \|\partial_t F\|_{L_t^1 L_x^4(I\times\R^4)}
\]
uniformly in $t_0\in I$.  Averaging over $t_0\in I$ and applying H\"older's inequality,
\begin{align*}
\|F\|_{L_t^\infty L_x^4(I\times\R^4)} &\leq |I|^{-1} \|F\|_{L_t^1 L_x^4(I\times\R^4)} + \|\partial_t F\|_{L_t^1 L_x^4(I\times\R^4)} \\
& \leq |I|^{-\frac14}\|F\|_{L_{t,x}^4(I\times\R^4)} + |I|^{\frac34}\|\partial_t F\|_{L_{t,x}^4(I\times\R^4)}.
\end{align*}
To pass to \eqref{E:gen}, we partition $\R$ into intervals of length $\delta$ and sum the fourth power of the inequality above over the partition.

Now we apply \eqref{E:gen} to $F=e^{it\Delta} P_{N} f^\omega$.  Using also Bernstein and \eqref{4,4}, we find
\begin{align*}
\E\bigl( \| e^{it\Delta} P_{N} f^\omega \|_{L_t^\infty L_x^4}^4\bigr) & \lesssim  (\delta^{-1} + \delta^3 N^8)\E\bigl( \|e^{it\Delta} P_{N} f^\omega\|_{L_{t,x}^4}^4\bigr) \\
& \lesssim (\delta^{-1}+ \delta^3 N^8)\|P_Nf\|_{L^2_x}^4 \\
& \lesssim (\delta^{-1}N^{-4s}+ \delta^3 N^{8-4s})\|f\|_{H_x^s}^4.
\end{align*}
Optimizing in $\delta$, we get
\[
\E\bigl(\|e^{it\Delta}P_Nf^\omega\|_{L_t^\infty L_x^4}^4\bigr) \lesssim N^{2-4s}\|f\|_{H_x^s}^4.
\]
Thus
\[
\|e^{it\Delta} f^\omega\|_{L_\omega^4 L_t^\infty L_x^4} \lesssim \sum_{N\geq 1} N^{\frac12-s}\|f\|_{H_x^s} \lesssim \|f\|_{H_x^s}
\]
whenever $s>\frac12$, which implies almost sure finiteness of the $L_t^\infty L_x^4$ norm.
\end{proof}

\begin{proposition}\label{P:STE} Fix $s>\frac 56$ and $f\in H^s_{\rad}(\R^4)$.  For $p$ sufficiently large,
\[
 \|\x^{\frac3p+\frac12} e^{it\Delta} f^\omega\|_{L_t^{\frac{2p}{p-2}}L_x^\infty} + \|\x^{\frac3p+\frac12}\nabla e^{it\Delta} f^\omega\|_{L_t^{\frac{2p}{p-2}}L_x^\infty} + \|\x^{\frac{1}{p}}\nabla e^{it\Delta} f^\omega\|_{L_t^{\frac{2p}{p-2}}L_x^4} <\infty
\]
almost surely, where all space-time norms are over $\R\times\R^4$.
\end{proposition}

We break the proof of Proposition~\ref{P:STE} into three lemmas, whose proofs all rely on applications of Lemma~\ref{L:weighted}, but with different exponents.

\begin{lemma}\label{L:1st-bd}  Let $4\leq p<\infty$ and $0\leq\beta\leq 1$. Then for any $s>0$ and $f\in H^s_{\rad}(\R^4)$,
\[
\E\biggl[\| \x^\beta e^{it\Delta} f^\omega\|_{L_t^{\frac{2p}{p-2}}L_x^\infty(\R\times\R^4)}^{\frac{2p}{p-2}}\biggr]\lesssim \|f\|_{H^s_x(\R^4)}^{\frac{2p}{p-2}}.
\]
\end{lemma}

\begin{proof} By Strichartz and Corollary~\ref{C:RS},
\[
\|e^{it\Delta}f\|_{L_t^{\frac{2p}{p-2}}L_x^{\frac{4p}{p+2}}} + \||x|^\beta e^{it\Delta} f\|_{L_t^{\frac{2p}{p-2}}L_x^{\frac{4p}{p+2-p\beta}}}\lesssim \|f\|_{L_x^2},
\]
provided $p\geq 4$ (to ensure that $\frac{2p}{p-2}\in(2,4]$) and $0\leq\beta\leq 1+\frac2p$.  Thus, for
\[
\max\{\tfrac{2p}{p-2},\tfrac{4p}{p+2-\beta p}\}\leq m <\infty,
\]
an application of Lemma~\ref{L:weighted} with $r_1 = \frac{4p}{p+2}$ and $r_2=\frac{4p}{p+2-p\beta}$ yields
\begin{equation}\label{E:mL2}
\E\biggl[\|\x^\beta e^{it\Delta}f^\omega\|_{L_t^{\frac{2p}{p-2}}L_x^m(\R\times\R^4)}^{\frac{2p}{p-2}}\biggr] \lesssim \|f\|_{L_x^2}^{\frac{2p}{p-2}}.
\end{equation}

Now given $s>0$, we may choose $m$ large enough so that we also have $s>\frac4m$.  Using \eqref{E:mL2} together with Lemma~\ref{L:commutator}, for $\beta \leq 1$ we find
\begin{align*}
\|\x^\beta e^{it\Delta} f^\omega\|_{L_{t,\omega}^{\frac{2p}{p-2}}L_x^\infty} & \lesssim \sum_{N\geq 1} N^{\frac4m}\|\x^\beta e^{it\Delta} P_N f^\omega\|_{L_{t,\omega}^{\frac{2p}{p-2}}L_x^m} \\
& \lesssim \sum_{N\geq 1} N^{\frac4m} \|P_N f\|_{L_x^2} \lesssim \|f\|_{H^s_x}.
\end{align*}
This completes the proof.\end{proof}

\begin{lemma}\label{L:2nd-bd} Fix $s>\frac56$ and $f\in H^s_{\rad}(\R^4)$.  For $p$ sufficiently large,
\[
\E\biggl[\| \x^{\frac12+\frac3p} e^{it\Delta}\nabla f^\omega\|_{L_t^{\frac{2p}{p-2}}L_x^\infty(\R\times\R^4)}^{\frac{2p}{p-2}}\biggr] \lesssim \|f\|_{H^s_x(\R^4)}^{\frac{2p}{p-2}}.
\]
\end{lemma}

\begin{proof} Let $\eps>0$ be a small parameter to be chosen below.  An application of Lemma~\ref{L:weighted} with $r_1=2$ and $r_2=\frac{2}{1-\theta}$, where
\[
\theta=\tfrac{2}{3+2\eps}(1+\tfrac1p +\eps)
\]
yields
\begin{align*}
\E\biggl[\| \x^{\frac12+\frac3p} &e^{it\Delta}\nabla P_N f^\omega\|_{L_t^{\frac{2p}{p-2}}L_x^m}^{\frac{2p}{p-2}}\biggr] \\
&\lesssim\|e^{it\Delta} \nabla P_N f\|_{L_t^{\frac{2p}{p-2}}L_x^2(\R\times B)}^{\frac{2p}{p-2}} + \||x|^{\frac12+\frac3p}e^{it\Delta}\nabla P_N f\|_{L_t^{\frac{2p}{p-2}}L_x^{r_2}(\R\times B^c)}^{\frac{2p}{p-2}},
\end{align*}
provided $m\geq \max\{2, \frac2{1-\theta}, \frac {2p}{p-2}\}$.

By H\"older's inequality, Lemma~\ref{L:LS}, and Bernstein's inequality:
\begin{equation}\label{E:LS1}
\begin{aligned}
\|e^{it\Delta} \nabla P_N f\|_{L_t^{\frac{2p}{p-2}}L_x^2(\R\times B)} & \lesssim \|e^{it\Delta}\nabla P_Nf\|_{L_{t,x}^2(\R\times B)}^{1-\frac2p} \|e^{it\Delta} \nabla P_Nf\|_{L_t^\infty L_x^2}^{\frac2p} \\
& \lesssim N^{\frac12+\frac1p}\|P_N f\|_{L_x^2}.
\end{aligned}
\end{equation}
On the other hand, setting
\[
q=\tfrac{2(p+1+p\eps)}{(p-2)(1+\eps)},
\]
we may apply H\"older's inequality, Proposition~\ref{P:RS}, and Lemma~\ref{L:LS} (provided we choose $0<\eps\ll 1$ and $p$ large) to get
\begin{equation}\label{E:2nd-bd}
\begin{aligned}
\| |x|^{\frac12+\frac3p}e^{it\Delta}\nabla &P_N f\|_{L_t^{\frac{2p}{p-2}}L_x^{r_2}(\R\times B^c)} \\
& \lesssim  \||x|^{\frac{2(q-1)}{q}}e^{it\Delta} \nabla P_N f\|_{L_t^q L_x^\infty}^\theta
\| \x^{-\frac12-\eps}e^{it\Delta} \nabla P_N f\|_{L_{t,x}^2}^{1-\theta} \\
& \lesssim N^{\frac{5+4\eps}{6+4\eps}+\frac{1}{p(3+2\eps)}}\|P_N f\|_{L_x^2}.
\end{aligned}
\end{equation}

Collecting \eqref{E:LS1} and \eqref{E:2nd-bd}, we find
\begin{equation}\label{E:mH1}
\E\biggl[\| \x^{\frac12+\frac3p} e^{it\Delta}\nabla P_N f^\omega\|_{L_t^{\frac{2p}{p-2}}L_x^m}^{\frac{2p}{p-2}}\biggr] \lesssim \bigl(N^{\frac{5+4\eps}{6+4\eps}+\frac{1}{p(3+2\eps)}}\|P_N f\|_{L_x^2}\bigr)^{\frac{2p}{p-2}}
\end{equation}
for $\eps$ small and $p, m$ large.

We now proceed as in the proof of Lemma~\ref{L:1st-bd}, using \eqref{E:mH1} together with Lemma~\ref{L:commutator} to estimate
\begin{align*}
\| \x^{\frac12+\frac3p} e^{it\Delta}\nabla f^\omega\|_{L_{t,\omega}^{\frac{2p}{p-2}}L_x^\infty} & \lesssim \sum_{N\geq 1} N^{\frac4m+\frac{5+4\eps}{6+4\eps}+\frac{1}{p(3+2\eps)}}\|P_N f\|_{L_x^2} \lesssim \|f\|_{H^s_x},
\end{align*}
provided
\[
s>\tfrac4m + \tfrac{5+4\eps}{6+4\eps}+\tfrac{1}{p(3+2\eps)}.
\]
Note that for any $s>\frac56$, we may choose $\eps$ sufficiently small and $p,m$ sufficiently large to guarantee that this condition holds.
\end{proof}

\begin{lemma}\label{L:3rd-bd} Fix $s>\frac23$ and $f\in H^s_{\rad}(\R^4)$. For $p$ sufficiently large,
\[
\E\biggl[\|\x^{\frac1p}e^{it\Delta}\nabla f^\omega\|_{L_t^{\frac{2p}{p-2}}L_x^4(\R\times\R^4)}^{\frac{2p}{p-2}}\biggr] \lesssim \|f\|_{H_x^s(\R^4)}^{\frac{2p}{p-2}}.
\]
\end{lemma}

\begin{proof} Let $\eps>0$ be a small parameter to be chosen below.  We apply \eqref{E:weighted2} with $r_1=2$ and $r_2=\frac{2}{1-\theta}$, where
\[
\theta = \tfrac{1}{3+2\eps}(1-\tfrac2p+2\eps).
\]
We estimate the contribution of $B$ using \eqref{E:LS1}.  The contribution of $B^c$ will be estimated as in \eqref{E:2nd-bd}, but with a different choice of exponents.  To be precise, we will now take
\[
q=\tfrac{2[p(1+2\eps)-2]}{p(1+2\eps)-8-4\eps},
\]
which belongs to the range $(2,4]$ for $\eps$ small and $p$ large.  Note that to apply \eqref{E:weighted2} also requires $r_2\leq 4$, which is also satisfied for $\eps$ small and $p$ large.  In this case, the contribution of $B^c$ can be estimated by
\[
\|e^{it\Delta}\nabla P_N f\|_{L_t^{\frac{2p}{p-2}}L_x^{r_2}(R\times B^c)} \lesssim N^{\frac{2+2\eps}{3+2\eps}-\frac{1}{p(3+2\eps)}}\|P_Nf\|_{L_x^2}.
\]

Choosing $\eps$ sufficiently small and $p$ sufficiently large, we can therefore estimate
\[
\|\x^{\frac1p}e^{it\Delta} \nabla f^\omega\|_{L_{t,\omega}^{\frac{2p}{p-2}}L_x^4} \lesssim \|f\|_{H^s_x}
\]
for any $s>\frac23$.
\end{proof}

Collecting the results of Lemmas~\ref{L:1st-bd}, \ref{L:2nd-bd}, and \ref{L:3rd-bd}, we obtain Proposition~\ref{P:STE}.

\section{Well-posedness and scattering for the forced equation}\label{S:wp}

In this section we prove well-posedness and a conditional scattering result for the forced NLS
\begin{equation}\label{fnls}
\begin{cases}
(i\partial_t + \Delta) v = |v+F|^2(v+F), \\
v(t_0)=v_0.
\end{cases}
\end{equation}
We will consider forcing terms $F$ satisfying $(i\partial_t+\Delta)F=0$ and the following bounds:
\begin{equation}\label{lwp-F}
F\in L_t^{\frac{4p}{p+2}}L_x^4(\R\times\R^4)\cap L_t^{\frac{2p}{p-2}}W_x^{1,\infty}(\R\times\R^4)
\end{equation}
for some large, but finite $p$. Note that by Propositions~\ref{P:STE2} and \ref{P:STE}, $F^\omega=e^{it\Delta}f^\omega$ satisfies \eqref{lwp-F} almost surely for $f\in H^s_{\rad}(\R^4)$ with $s>\frac56$ and $p$ sufficiently large.

\begin{proposition}[Local well-posedness]\label{P:lwp} Let $t_0\in\R$, $v_0\in H^1(\R^4)$, and $F$ be a solution to $(i\partial_t + \Delta) F = 0$ satisfying \eqref{lwp-F}.  Suppose $\|v_0\|_{H^1_x}\leq E$.  There exists $\eta_0=\eta_0(E)>0$ so that if $I\ni t_0$ is an open interval such that
\begin{equation}\label{lwp-condition}
\|e^{i(t-t_0)\Delta}v_0\|_{L_t^4 L_x^8(I\times\R^4)} + \|F\|_{L_t^{\frac{4p}{p+2}}L_x^4\cap L_t^{\frac{2p}{p-2}}W_x^{1,\infty}(I\times\R^4)} \leq\eta\leq\eta_0,
\end{equation}
then there exists a unique solution $v\in C_tH^1_x\cap L_t^4L_x^8(I\times\R^4)$ to \eqref{fnls} on $I$.  In particular, for any $v_0\in H^1(\R^4)$ there exists a unique local-in-time solution $v$ to \eqref{fnls}, which extends to a maximal lifespan $\Imax$.

Moreover, if $F(t_0)\in L_x^2$, we have the following blowup/scattering criterion:
\begin{itemize}
\item[(i)] If $\sup\Imax<\infty$, then $\|v\|_{L_t^4 L_x^8((t_0,\sup\Imax)\times\R^4)}=\infty$.
\item[(ii)] If $\sup\Imax=\infty$ and $\|v\|_{L_t^4 L_x^8((t_0,\infty)\times\R^4)}<\infty$, then $v$ scatters forward in time.
\end{itemize}
The analogous statements hold backward in time.

\end{proposition}

\begin{proof}  Without loss of generality assume $t_0=0$.  Define
\[
[\Phi v](t) = e^{it\Delta}v_0 - i\int_0^t e^{i(t-s)\Delta}\bigl[|v(s)+F(s)|^2(v(s)+F(s))\bigr]\,ds.
\]
Let $\eta>0$ to be chosen below and let $I\ni 0$ be a time interval as in \eqref{lwp-condition}. Note that for any $v_0\in H^1_x$, such an interval exists by Sobolev embedding, Strichartz estimates, and the monotone convergence theorem.

In the following, we take space-time norms over $I\times\R^4$.  Define
\[
X=\{v:I\times\R^4\to \C:\,\|v\|_{L_t^\infty H_x^1} \leq 2CE,\quad \|v\|_{L_t^4 L_x^8}\leq 2C\eta\}.
\]
Here $C$ is a constant that accounts for implicit constants appearing in Strichartz estimates, Sobolev embedding, etc.  We equip $X$ with the $L_t^\infty L_x^2$ metric.

We write
\[
|v+F|^2(v+F)=|v|^2 v + |F|^2 F + 2|v|^2 F + v^2\bar F + 2|F|^2 v + F^2 \bar v.
\]
To estimate the nonlinearity, we note that for $p>6$ the pair $(\frac{4p}{3p-6},\frac{8p}{5p+6})$ is dual admissible, while for $p>2$ the pair $(\frac{p+6}{p+2},\frac{2(p+6)}{p+10})$ is also dual admissible.  Using the product rule and H\"older's inequality, we estimate
\begin{align*}
\| \jb(|v|^2 v)\|_{L_t^2 L_x^{\frac43}}&\lesssim \|v\|_{L_t^4 L_x^8}^2 \|\jb v\|_{L_t^\infty L_x^2}, \\
\| \jb(|F|^2 F)\|_{L_t^1 L_x^2} &\lesssim \| F\|_{L_t^{\frac{4p}{p+2}}L_x^4}^2 \|F\|_{L_t^{\frac{2p}{p-2}}W_x^{1,\infty}}, \\
\| \jb(Fv^2)\|_{L_t^{\frac{4p}{3p-6}}L_x^{\frac{8p}{5p+6}}}&\lesssim \|F\|_{L_t^{\frac{2p}{p-2}}W_x^{1,\infty}}\|\jb v\|_{L_t^\infty L_x^2}\|v\|_{L_t^\infty L_x^2}^{\frac2p}\|v\|_{L_t^4 L_x^8}^{1-\frac2p}, \\
\| \jb(F^2 v)\|_{L_t^{\frac{p+6}{p+2}}L_x^{\frac{2(p+6)}{p+10}}}& \lesssim \|F\|_{L_t^{\frac{2p}{p-2}}W_x^{1,\infty}}^{\frac{2p+4}{p+6}}\|F\|_{L_t^{\frac{4p}{p+2}}L_x^4}^{\frac{8}{p+6}}\|\jb v\|_{L_t^\infty L_x^2}.
\end{align*}
Thus, an application of Strichartz shows that for $v\in X$,
\begin{align*}
\|\Phi v\|_{L_t^\infty H_x^1}\lesssim E + \eta^2 E + \eta^3+ \eta^{2-\frac2p} E^{1+\frac2p} \leq 2CE
\end{align*}
for $\eta\leq\eta_0(E)$ small. Similarly, using $\dot H^{1,\frac83}(\R^4)\hookrightarrow L^8(\R^4)$, we have
\[
\|\Phi v\|_{L_t^4 L_x^8} \lesssim \eta +  \eta^2 E + \eta^3+ \eta^{2-\frac2p} E^{1+\frac2p} \leq 2C\eta
\]
for $\eta\leq\eta_0(E)$ small.  Thus $\Phi:X\to X$.

Next, note that
\[
\bigl||v+F|^2(v+F)-|w+F|^2(w+F)\bigr| \lesssim |v-w|\bigl(|v|^2+|w|^2+|F|^2).
\]
Estimating essentially as above, we find
\[
\|\Phi v - \Phi w \|_{L_t^\infty L_x^2} \lesssim \eta^{2}\|v-w\|_{L_t^\infty L_x^2}
\]
for any $v,w\in X$.  Thus $\Phi$ is a contraction for $\eta\leq\eta_0(E)$ small and we deduce the existence of a solution on $I$, which may then be extended to its maximal lifespan $\Imax$.

Note that since $F$ solves the linear Schr\"odinger equation, $u:= F+v$ solves \eqref{nls} on $\Imax\times\R^4$ with $u(0)=v_0+F(0)$. Thus if $F(0)\in L_x^2$, then by the conservation of mass for \eqref{nls} and the triangle inequality we get
\begin{align}\label{mass v}
\|v\|_{L_t^\infty L_x^2(\Imax\times\R^4)}\lesssim \|v_0\|_{L_x^2} + \|F(0)\|_{L_x^2}.
\end{align}

Next suppose toward a contradiction that $\sup \Imax <\infty$ but
\begin{equation}\label{bs-contra1}
\|v\|_{L_t^4 L_x^8((0,\sup\Imax)\times\R^4)}<\infty.
\end{equation}
Fix $\eps>0$ to be chosen below.  Using \eqref{bs-contra1} and \eqref{lwp-F}, we may decompose $(0,\sup \Imax)$ into finitely many intervals $I_j$ so that
\[
\|v\|_{L_t^4 L_x^8(I_j\times\R^4)} +  \|F\|_{L_t^{\frac{4p}{p+2}}L_x^4\cap L_t^{\frac{2p}{p-2}}W_x^{1,\infty}(I_j\times\R^4)}<\eps
\]
for each $j$.  Using the nonlinear estimates above, we find
\[
\bigl|\|v\|_{L_t^\infty H_x^1(I_1\times\R^4)}-\|v_0\|_{H^1_x} \bigr|\lesssim + \eps^3 + (\eps^2 + \eps^{2-\frac2p}\|v\|_{L_t^\infty L_x^2(I_1\times\R^4)}^{\frac2p})\|v\|_{L_t^\infty H_x^1(I_1\times\R^4)}.
\]
Thus, recalling \eqref{mass v} and choosing $\eps$ sufficiently small compared to $\|v_0\|_{H^1_x}$ and $\|F(0)\|_{L_x^2}$, we deduce
\[
\|v\|_{L_t^\infty H_x^1(I_1\times\R^4)} \leq 2\|v_0\|_{H^1_x}.
\]
We can repeat this argument on $I_2$ (with the same choice of $\eps$) to deduce a bound of $4\|v_0\|_{H^1}$.  By induction,
\[
\|v\|_{L_t^\infty H_x^1((0,\sup\Imax)\times\R^4)} \lesssim 2^{C(\eps)}\|v_0\|_{H^1_x}.
\]
Using this bound and \eqref{bs-contra1}, the nonlinear estimates then imply
\[
\biggl\|\int_{t_0}^t e^{i(t-s)\Delta}\bigl[|v(s)+F(s)|^2(v(s)+F(s)\bigr]\,ds\biggr\|_{L_t^4 L_x^8((t_0,\sup\Imax)\times\R^4)}\lesssim 1
\]
uniformly in $t_0\in(0,\sup\Imax)$.  Thus, by the Duhamel formula, the triangle inequality, and monotone convergence,
\[
\lim_{t_0\to \sup\Imax}\|e^{i(t-t_0)\Delta}v(t_0)\|_{L_t^4 L_x^8((t_0,\sup\Imax)\times\R^4)}=0.
\]
In particular, there exists $\delta>0$ so that
\[
\|e^{i(t-t_0)\Delta}v(t_0)\|_{L_t^4 L_x^8((t_0-\delta,\sup\Imax+\delta)\times\R^4)}<\tfrac12\eta_0,
\]
where $\eta_0=\eta_0(\|v\|_{L_t^\infty H_x^1((0,\sup\Imax)\times\R^4)})$ is the same as in the statement of local well-posedness. However, this implies that the solution $v$ extends beyond $\sup\Imax$, a contradiction.

Finally, suppose that $\sup\Imax=\infty$ and $v\in L_t^4 L_x^8((0,\infty)\times\R^4)$.  Repeating the arguments just given, we can deduce that $v\in L_t^\infty H_x^1((0,\infty)\times\R^4)$.  An application of Strichartz combined with the observation that
\[
 \|v\|_{L_t^4 L_x^8((s,t)\times\R^4)} + \|F\|_{L_t^{\frac{2p}{p-2}}W_x^{1,\infty}((s,t)\times\R^4)}\to 0\qtq{as}s,t\to\infty,
\]
yields that $e^{-it\Delta}v(t)$ is Cauchy in $H_x^1$ as $t\to\infty$.
\end{proof}

Our next goal is a conditional scattering result for \eqref{fnls}; see Proposition~\ref{P:v-scatter}.  As described in the introduction, this relies on a stability theory for \eqref{fnls}, which we elaborate next. 

\begin{lemma}[Short-time stability]\label{L:stab} Let $I\ni t_0$ be a time interval and let $v_0\in H^1(\R^4)$ with $\|v_0\|_{H_x^1} \leq E$. Suppose $v:I\times\R^4\to \C$ is a solution to \eqref{fnls} where $v(t_0)=v_0$ and $F$ is a solution to $(i\partial_t+\Delta)F=0$ satisfying \eqref{lwp-F}.  Suppose $u_0\in H^1(\R^4)$ satisfies
\[
\|v_0-u_0\|_{H^1_x}\leq \eps
\]
for some $0<\eps<\eps_0$.  Let $u$ be the solution to \eqref{nls} with $u(t_0)=u_0$ and suppose
\[
\|u\|_{L_t^4 L_x^8(I\times\R^4)}\leq\delta.
\]
Finally, suppose
\[
\|F\|_{L_t^{\frac{2p}{p-2}}W_x^{1,\infty}(I\times\R^4)} + \|F\|_{L_t^{\frac{4p}{p+2}}L_x^4(I\times\R^4)} \leq\eps.
\]
Then for $\eps_0, \delta$ sufficiently small depending on $E$,
\[
\|\jb(v-u)\|_{L_t^\infty L_x^2 \cap L_t^4 L_x^{\frac83}(I\times\R^4)} \leq C(E)\eps.
\]
\end{lemma}

\begin{proof} Without loss of generality, assume $t_0=0=\inf I$. In the following, all space-time norms are taken over $I\times\R^4$.  Define $w=v-u$ and set $S=L_t^\infty L_x^2 \cap L_t^4 L_x^{\frac83}$.

Standard continuity arguments combined with an application of the Strichartz inequality yield
\[
\|\jb u\|_S \lesssim E,
\]
for $\eps_0,\delta$ sufficiently small depending on $E$.  Thus, using the equation for $w$, the nonlinear estimates from the local theory, and the hypotheses of the lemma, we get
\begin{align*}
\|\jb w\|_S & \lesssim \|v_0\!-\!u_0\|_{H^1_x} + \eps^3 + \eps^2\|\jb v\|_S + \eps\|\jb v\|_S^2 + \|\jb\bigl(|v|^2 v - |u|^2 u\bigr)\|_{L_t^2 L_x^{\frac43}} \\
& \lesssim C(E)\eps + \eps^2\|\jb w\|_S + \eps\|\jb w\|_S^2 + \|\jb\bigl(|v|^2 v\!\!-\!\!|u|^2 u\bigr)\|_{L_t^2 L_x^{\frac43}}.
\end{align*}
Using $L_t^4 \dot H_x^{1,\frac83}\hookrightarrow L_t^4 L_x^8$,
\begin{align*}
\|\jb\bigl(|v|^2 v - |u|^2 u)\|_{L_t^2 L_x^{\frac43}} & \lesssim \|\jb w\|_{S}\bigl(\|\jb w\|_S^2 + \|u\|_{L_t^4 L_x^8}\|\jb u\|_{S}\bigr) \\
& \lesssim \delta E \|\jb w\|_{S} + \|\jb w\|_S^3.
\end{align*}
Combining the estimates above and choosing $\delta,\eps_0$ sufficiently small depending on $E$, a standard continuity argument yields the result.
\end{proof}

\begin{lemma}[Long-time stability]\label{L:stab2} Let $I\ni t_0$ be a time interval and let $v_0\in H^1(\R^4)$ with $\|v_0\|_{H^1_x}\leq E$. Let $u$ be the solution to \eqref{nls} with $u(t_0)=v_0$.  Suppose that
\[
\|u\|_{L_t^4 L_x^8(I\times\R^4)} \leq L.
\]
Then there exists $\eps_1=\eps_1(E,L)>0$ so that if $F$ is a solution to $(i\partial_t+\Delta)F=0$ satisfying
\begin{equation}\label{stab-F}
\|F\|_{L_t^{\frac{2p}{p-2}}W_x^{1,\infty}(I\times\R^4)} + \|F\|_{L_t^{\frac{4p}{p+2}}L_x^4(I\times\R^4)} \leq\eps
\end{equation}
for some $0<\eps\leq\eps_1$, then there exists a unique solution $v$ to \eqref{fnls} with $v(t_0)=v_0$ on $I\times\R^4$.  Moreover,
\[
\|v\|_{L_t^4 L_x^8(I\times\R^4)}\leq C(E,L)<\infty.
\]
\end{lemma}

\begin{proof} Assume without loss of generality that $t_0=0=\inf I$.  By the local theory, it suffices to establish the $L_t^4 L_x^8$ bound for $v$ as an \emph{a priori} estimate.  Note that by conservation of mass and energy, we may assume $\|u\|_{L_t^\infty H_x^1} \leq C_0(E)$.

Choose $\delta = \delta(2C_0(E))$ as in Lemma~\ref{L:stab} and divide $I$ into $J=J(E,L)$ subintervals $I_j=[t_j,t_{j+1}]$ so that
\[
\|u\|_{L_t^4 L_x^8(I_j\times\R^4)} \leq \delta
\]
for each $j$.

We claim that if we choose $\eps_1=\eps_1(E,L)$ sufficiently small and assume \eqref{stab-F}, then there exists $C_j\geq 1$ so that
\begin{equation}\label{stab-induction}
\|v(t_j)-u(t_j)\|_{H_x^1}\leq C_j\eps\leq\eps_0\qtq{and}\|v(t_j)\|_{H_x^1}\leq 2C_0(E)\quad\text{for each $j$,}
\end{equation}
where $\eps_0=\eps_0(2C_0(E))$ is as in Lemma~\ref{L:stab}.

Note that \eqref{stab-induction} holds trivially for $j=0$.  Now suppose it holds for each $0\leq k\leq j-1$; we will prove it holds at $j$.  Using the Duhamel formula and the inductive hypothesis, and estimating as in Lemma~\ref{L:stab}, we get
\begin{align*}
\|v(t_j)-u(t_j)\|_{H_x^1}  \lesssim \||v+F|^2(v+F)-|u|^2 u\|_{N([0,t_j])} \leq C(E)\sum_{k=0}^{j-1} C_k\eps.
\end{align*}
Thus we may define $C_j$ inductively with $C_0=1$ and $C_j=C(E)\sum_{k=0}^{j-1}C_k$.  Choosing $\eps_1=\eps_1(E,L)$ sufficiently small, we can also ensure that
\[
\sup_{0\leq j\leq J} C_j\eps \leq\eps_0(2C_0(E)).
\]
Then by the triangle inequality,
\[
\|v(t_j)\|_{H_x^1(\R^4)} \leq C_0(E)+C_j\eps \leq 2C_0(E),
\]
for $\eps\leq\eps_1(E,L)$ small enough.  This completes the induction and settles \eqref{stab-induction}.

We may therefore apply Lemma~\ref{L:stab} on each $I_j$, yielding $L_t^4 L_x^8$ bounds for $v$.  Summing up these bounds completes the proof.  \end{proof}

With Lemma~\ref{L:stab2} in hand, we are now in a position to prove the following: 

\begin{proposition}[$H^1_x$ bounds imply scattering]\label{P:v-scatter} Let $v_0\in H^1(\R^4)$ and let $F$ be a solution to $(i\partial_t+\Delta) F = 0$ satisfying \eqref{lwp-F}.  Let $v:\Imax\times\R^4\to\C$ be the maximal-lifespan solution to \eqref{fnls} with $v(0)=v_0$.  Suppose
\begin{equation}\label{bounded-energy}
\sup_{t\in(0,\sup\Imax)}\|v(t)\|_{H_x^1}\leq E<\infty.
\end{equation}
Then $\sup\Imax=\infty$ and $v$ scatters as $t\to\infty$.  The analogous statements hold backward in time.
\end{proposition}

\begin{proof} By Proposition~\ref{P:lwp}, it suffices to show that
\[
\|v\|_{L_t^4 L_x^8((0,\sup\Imax)\times\R^4)} \leq C(E).
\]
To prove this, we will rely on Theorem~\ref{T:EC}, which guarantees that there exists a unique global solution $u$ to \eqref{nls} from data $\|u(t_0)\|_{H_x^1}\leq E$ and it satisfies
\[
\|u\|_{L_t^4 L_x^8(\R\times\R^4)}\leq L(E).
\]
Now let $\eps_1=\eps_1(E,L(E))$ be as in Lemma~\ref{L:stab2} and divide $(0,\sup\Imax)$ into finitely many intervals $\{I_j\}_{j=0}^J$ so that
\[
\|F\|_{L_t^{\frac{2p}{p-2}}W_x^{1,\infty}(I_j\times\R^4)} + \|F\|_{L_t^{\frac{4p}{p+2}}L_x^4(I_j\times\R^4)} \leq\eps_1
\]
for each $j$. Note that $J=J(E)$, so that it suffices to show
\[
\|v\|_{L_t^4 L_x^8(I_j\times\R^4)}\leq C(E)\qtq{for each}j.
\]
To this end, write $I_j=[t_j,t_{j+1}]$. Let $u$ be the solution to \eqref{nls} with initial data $v(t_j)$.  Then since $\|v(t_j)\|_{H_x^1}\leq E$, we have
\[
\|u\|_{L_t^4 L_x^8(I_j\times\R^4)}\leq\|u\|_{L_t^4 L_x^8(\R\times\R^4)}\leq L(E).
\]
Thus we are in a position to apply the stability result Lemma~\ref{L:stab2}, yielding
\[
\|v\|_{L_t^4 L_x^8(I_j\times\R^4)}\leq C(E,L(E)) = C(E),
\]
as needed. \end{proof}

\section{Energy bounds for the forced equation}\label{S:E}

In this section, we prove that suitable space-time bounds on the forcing term $F$ guarantee that the solution $v$ to the forced equation \eqref{fnls} obeys uniform energy bounds, and so one may invoke Proposition~\ref{P:v-scatter} to conclude that $v$ scatters in $H^1_x$.  The particular norm we rely on is
\begin{align}
\|F\|_{X(I)}& :=  \|\x^{\frac3p+\frac12} F\|_{L_t^{\frac{2p}{p-2}}L_x^\infty} +  \|\x^{\frac3p+\frac12}\nabla F\|_{L_t^{\frac{2p}{p-2}}L_x^\infty} + \|\x^{\frac{1}{p}}\nabla F\|_{L_t^{\frac{2p}{p-2}}L_x^4}\notag\\
& \quad+ \|F\|_{L_t^3 L_x^6} +\|F\|_{L_t^{\frac{4p}{p+2}}L_x^4} + \|F\|_{L_{t,x}^4},\label{X}
\end{align}
where all space-time norms are over $I\times\R^4$ and $p$ is large, but finite.  Note that for $F^\omega=e^{it\Delta}f^\omega$, Propositions~\ref{P:STE2} and \ref{P:STE} guarantee that $\|F^\omega\|_{X(\R)}<\infty$ almost surely, whenever $f\in H^s_{\rad}(\R^4)$ for $s>\frac56$ and $p$ is taken sufficiently large.

Our main result in this section is the following:

\begin{proposition}[Energy bounds]\label{P:energy-bds} Suppose that $F$ is a solution to $(i\partial_t+\Delta) F = 0$ satisfying
\begin{align}\label{3:30}
\|F\|_{X(\R)}+\|F\|_{L_t^\infty L_x^2(\R\times\R^4)}+\|F\|_{L_t^\infty L_x^4(\R\times\R^4)} < \infty.
\end{align}
Let $v_0\in H^1(\R^4)$ and let $v:\Imax\times\R^4\to\C$ be the maximal-lifespan solution to \eqref{fnls} with $v(0)=v_0$.  Then $\sup_{t\in\Imax}\|v(t)\|_{H_x^1} <\infty$.
\end{proposition}

By time-reversal symmetry, it suffices to prove uniform energy bounds for $v$ on $[0, \sup\Imax)$.  For $0<T\in \Imax$, we define
\begin{equation}\label{bold-E}
\EE(T) = \sup_{t\in[0,T]}E[v(t)],
\end{equation}
where the energy $E[\cdot]$ is as in \eqref{def:E}.  We seek bounds on $\EE(T)$ that are uniform in $T$.  We will prove this using a double bootstrap argument involving both a Morawetz inequality for $v$ and control of the energy increment for $v$.

\begin{lemma}[Morawetz estimate]\label{P:Morawetz} Suppose $v:[0,T]\times\R^4\to\C$ is a solution to \eqref{fnls} satisfying the uniform mass bound
\begin{equation}\label{mass-bd}
\|v\|_{L_t^\infty L_x^2([0,T]\times\R^4)}\lesssim 1.
\end{equation}
Writing
\begin{equation}\label{def:A}
\begin{aligned}
A(T)&:=\| \x^{-\frac14}v\|_{L_{t,x}^4([0,T]\times\R^4)}^4 + \|\x^{-\frac32}v\|_{L_{t,x}^2([0,T]\times\R^4)}^2\\ & \quad +\|\x^{-\frac32}\nabla v\|_{L_{t,x}^2([0,T]\times\R^4)}^2,
\end{aligned}
\end{equation}
we have
\begin{equation}\label{E:Morawetz2}
\begin{aligned}
A(T)& \lesssim \EE(T)^{\frac12}+\EE(T)^{\frac12}\|F\|_{X([0,T])}^3+\EE(T)\|F\|_{X([0,T])}^{\frac{2p}{p-2}},
\end{aligned}
\end{equation}
where $\EE(\cdot)$ is as in \eqref{bold-E}.
\end{lemma}

\begin{proof}
We write \eqref{fnls} in the following form:
\[
(i\partial_t+\Delta) v = |v|^2 v + \N,\qtq{where} \N := |v+F|^2(v+F)-|v|^2 v.
\]

Given a weight $a=a(x)$, we define the standard Morawetz action
\[
m(t) =2\Im\int a_k(x) v_k(t,x) \bar v(t,x)\,dx,
\]
where subscripts denote derivatives and repeated indices are summed.  A direct computation using the equation and integration by parts leads to the Morawetz identity
\begin{align*}
\dot m(t) = \int -\Delta\Delta a|v|^2 + 4\Re a_{jk} \bar v_j v_k + \Delta a|v|^4 + 4a_k\Re\{\bar\N v_k\} + 2\Delta a\Re \{\bar v \N\}\,dx.
\end{align*}

For the weight $a(x)=\x$, one has
\[
\nabla a = \tfrac{x}{\x}, \quad a_{jk} = \tfrac{\delta_{jk}}{\x}-\tfrac{x_jx_k}{\x^3},\quad \Delta a = \tfrac{3}{\x} + \tfrac{1}{\x^3}, \quad -\Delta\Delta a = \tfrac{3}{\x^3}+\tfrac{6}{\x^5}+\tfrac{15}{\x^7}.
\]
Using Cauchy-Schwarz and \eqref{mass-bd}, we see that
\begin{equation}\label{mor-bd}
\|m\|_{L_t^\infty([0,T])}\lesssim \EE(T)^{\frac12}.
\end{equation}
Noting that
\[
\Re a_{jk} \bar v_j v_k \geq \x^{-3}|\nabla v|^2,
\]
we apply the Morawetz identity and the fundamental theorem of calculus to obtain
\begin{equation}\label{E:Morawetz1}
\begin{aligned}
&A(T) \lesssim \EE(T)^{\frac12} + \|\N\nabla v\|_{L_{t,x}^1([0,T]\times\R^4)}+ \|\x^{-1}\N v\|_{L_{t,x}^1([0,T]\times\R^4)}.
\end{aligned}
\end{equation}

To estimate the last two terms, we first note that by H\"older's inequality,
\begin{equation}\label{interpolate-mor}
\|\x^{-\frac3p}\nabla v\|_{L_t^p L_x^2} \lesssim \|\nabla v\|_{L_t^\infty L_x^2}^{1-\frac2p}\|\x^{-\frac32}\nabla v\|_{L_{t,x}^2}^{\frac2p}
\lesssim \EE(T)^{\frac12-\frac1p}A(T)^{\frac1p}
\end{equation}
for any $2\leq p\leq\infty$.  Using also that
\[
\N = |v+F|^2(v+F)-|v|^2 v = \mathcal{O}\bigl( Fv^2+ F^3\bigr),
\]
together with H\"older's and Hardy's inequalities, we estimate
\begin{align*}
 \|&\N\nabla v\|_{L_{t,x}^1}+ \|\x^{-1}\N v\|_{L_{t,x}^1} \\ & \lesssim \|\x^{-\frac3p}\nabla v\|_{L_t^p L_x^2} \|\x^{-\frac14}v\|_{L_{t,x}^4}^2 \|\x^{\frac3p+\frac12}F\|_{L_t^{\frac{2p}{p-2}}L_x^\infty} \\
& \quad + \|\x^{-\frac14}v\|_{L_{t,x}^4}^2\|\x^{-\frac32}v\|_{L_{t,x}^2}^{\frac13}\|v\|_{L_t^\infty L_x^4}^{\frac23}\|F\|_{L_t^3 L_x^6} + \|\nabla v\|_{L_t^\infty L_x^2} \|F\|_{L_t^3 L_x^6}^3 \\
& \lesssim \EE(T)^{\frac12-\frac1p}A(T)^{\frac12+\frac1p}\|F\|_{X([0,T])}+\EE(T)^{\frac16}A(T)^{\frac23}\|F\|_{X([0,T])} + \EE(T)^{\frac12}\|F\|_{X([0,T])}^3.
\end{align*}
Continuing from \eqref{E:Morawetz1} and using Young's inequality to absorb $A(T)$ into the left-hand side, we deduce \eqref{E:Morawetz2}.\end{proof}

\begin{lemma}[Energy increment]\label{P:energy}
Suppose $v:[0,T]\times\R^4\to\C$ is a solution to \eqref{fnls}. Then
\begin{equation}\label{E:E}
\begin{aligned}
\EE(T) & \lesssim E[v(0)] + \|F\|_{L_t^\infty L_x^4}^4 + A(T)\|F\|_{X([0,T])}^{\frac{2p}{p+2}} +A(T)^{\frac{p+4}{2(p+2)}}\|F\|_{X([0,T])}^{\frac{4p}{p+2}} \\
& \quad +A(T)^{\frac{2}{p+2}}\|F\|_{X([0,T])}^{\frac{6p}{p+2}} + A(T)^{\frac{4}{p+4}}\|F\|_{X([0,T])}^{\frac{4p}{p+4}} \\
& \quad + A(T)^{\frac{8}{3p+8}}\bigl[ \|F\|_{X([0,T])}^2\|F\|_{L_t^\infty L_x^4([0,T]\times\R^4)}\bigr]^{\frac{4p}{3p+8}},
\end{aligned}
\end{equation}
where $\EE(\cdot)$ is as in \eqref{bold-E} and $A(\cdot)$ is as in \eqref{def:A}.
\end{lemma}

\begin{proof} Set $G(z)=|z|^2 z$. A direct computation using \eqref{fnls} yields
\begin{align*}
\partial_t E[v(t)] &=-\Re\int [\partial_t \bar v] [G(v+F)-G(v)]\,dx \\
& =-\tfrac14\partial_t\int \bigl[|v+F|^4 - |v|^4-|F|^4]\,dx + \Re\int [\bar G(v+F)-\bar G(F)]\partial_t F\,dx,
\end{align*}
where in the last line we used the identity $\partial_t |z|^4 = 4\Re G(z)\partial_t \bar z$.  Recalling that $F$ solves $(i\partial_t+\Delta) F=0$, we continue from above and integrate by parts to get
\begin{align*}
\partial_t E[v(t)] = -\tfrac14\partial_t\! \int\bigl[v+F|^4\! -\! |v|^4\!-\!|F|^4]\,dx + \Im\int \nabla [\bar G(v+F)-\bar G(F)]\cdot\nabla F\,dx.
\end{align*}
In particular, by the fundamental theorem of calculus,
\begin{align*}
\EE[T] & \leq E[v(0)]+\bigl\|\, |v+F|^4 - |v|^4 - |F|^4\bigr\|_{L_t^\infty L_x^1([0,T]\times\R^4)}\\
&\quad +\bigl\| \nabla F\cdot \nabla\bigl[|v+F|^2(v+F)-|F|^2 F\bigr]\bigr\|_{L_{t,x}^1([0,T]\times\R^4)}.
\end{align*}

We first estimate the boundary terms:
\begin{align*}
\bigl\|\, |v+F|^4 - |v|^4 - |F|^4\bigr\|_{L_t^\infty L_x^1}&\lesssim \|v\|_{L_t^\infty L_x^4}\|F\|_{L_t^\infty L_x^4}^3 + \|v\|_{L_t^\infty L_x^4}^3 \|F\|_{L_t^\infty L_x^4} \\
& \lesssim \EE(T)^{\frac14}\|F\|_{L_t^\infty L_x^4}^3 + \EE(T)^{\frac34}\|F\|_{L_t^\infty L_x^4}.
\end{align*}

Distributing the derivative in the remaining term, we are led to estimate five terms.  Using H\"older's inequality and \eqref{interpolate-mor}, we obtain
\begin{align*}
\|v^2\nabla F\cdot\nabla v\|_{L_{t,x}^1} & \lesssim \|\x^{-\frac3p}\nabla v\|_{L_t^p L_x^2} \|\x^{-\frac14}v\|_{L_{t,x}^4}^2 \|\x^{\frac3p+\frac12}\nabla F\|_{L_t^{\frac{2p}{p-2}}L_x^\infty}  \\
& \lesssim A(T)^{\frac12+\frac1p}\EE(T)^{\frac12-\frac1p}\|F\|_{X([0,T])}, \\
\| vF\nabla F\cdot\nabla v \|_{L_{t,x}^1} & \lesssim \|\x^{-\frac3p}\nabla v\|_{L_t^p L_x^2} \|\x^{-\frac14}v\|_{L_{t,x}^4} \|\x^{\frac14+\frac3p}\nabla F\|_{L_t^{\frac{2p}{p-2}}L_x^\infty} \|F\|_{L_{t,x}^4} \\
& \lesssim A(T)^{\frac14+\frac1p}\EE(T)^{\frac12-\frac1p}\|F\|_{X([0,T])}^2,\\
\|F^2\nabla F \cdot\nabla v \|_{L_{t,x}^1}&\lesssim \|\x^{-\frac3p}\nabla v\|_{L_t^p L_x^2} \|F\|_{L_{t,x}^4}^2 \|\x^{\frac3p}\nabla F\|_{L_t^{\frac{2p}{p-2}}L_x^\infty} \\
& \lesssim A(T)^{\frac1p}\EE(T)^{\frac12-\frac1p}\|F\|_{X([0,T])}^3,\\
\|v^2|\nabla F|^2 \|_{L_{t,x}^1} & \lesssim \| \x^{-\frac14}v\|_{L_{t,x}^4}^{\frac8p}\|v\|_{L_t^\infty L_x^4}^{2-\frac8p} \|\x^{\frac1p}\nabla F\|_{L_t^{\frac{2p}{p-2}}L_x^4}^2 \\
& \lesssim A(T)^{\frac2p}\EE(T)^{\frac12-\frac2p}\|F\|_{X([0,T])}^2,\\
\|vF|\nabla F|^2  \|_{L_{t,x}^1} & \lesssim \|\x^{-\frac14}v\|_{L_{t,x}^4}^{\frac8p}\|v\|_{L_t^\infty L_x^4}^{1-\frac8p} \|F\|_{L_t^\infty L_x^4} \|\x^{\frac1p}\nabla F\|_{L_t^{\frac{2p}{p-2}}L_x^4}^2 \\
& \lesssim A(T)^{\frac2p} \EE(T)^{\frac14-\frac2p} \|F\|_{X([0,T])}^2\|F\|_{L_t^\infty L_x^4}.
\end{align*}
Collecting the estimates above and applying Young's inequality to absorb $\EE(T)$ into the left-hand side, we arrive at \eqref{E:E}.\end{proof}

We are now ready to present the proof of Proposition~\ref{P:energy-bds}.

\begin{proof}[Proof of Proposition~\ref{P:energy-bds}] First, as $F\in L_t^\infty L_x^2$, conservation of mass for \eqref{nls} implies that $v$ satisfies the mass bound \eqref{mass-bd}.

As remarked before, by time-reversal symmetry, it suffices to prove uniform energy bounds for $v$ on $[0, \sup \Imax)$.  To this end, let $0<\eta\ll1$ be a small parameter and subdivide $[0,\sup\Imax)$ into finitely many intervals $\{I_j\}_{j=0}^J$ so that
\begin{equation}\label{Fsmallj}
\|F\|_{X(I_j)} \leq \eta\qtq{for each}j.
\end{equation}
Inserting \eqref{Fsmallj} into \eqref{E:Morawetz2} and \eqref{E:E}, we find
\begin{align*}
A(T)&\lesssim \EE(T)^{\frac12}+\eta^{\frac{2p}{p-2}}\EE(T),\\
\EE(T)&\lesssim E[v(0)] + \|F\|_{L_t^\infty L_x^4}^4+ \eta^{\frac{2p}{p+2}}\bigl[ A(T) +A(T)^{\frac2{p+2}}\bigr] \\ & \quad + \bigl[\|F\|_{L_t^\infty L_x^4}\eta^2\bigr]^{\frac{4p}{3p+8}} A(T)^{\frac8{3p+8}},
\end{align*}
for all $0<T\in I_0$.  Recalling \eqref{3:30} and choosing $\eta$ sufficiently small, a continuity argument shows that $\EE$ can increase by at most a fixed constant on the interval $I_0$.  Repeating this argument on each $I_j$, we conclude that there exists a uniform energy bound on the forward maximal-lifespan of $v$.
\end{proof}

We now have all the pieces we need to complete the proof of our main result.

\begin{proof}[Proof of Theorem~\ref{T}] Let $s>\frac56$ and let $f\in H^s_{\rad}(\R^4)$ be radial.  Propositions~\ref{P:STE2} and Proposition~\ref{P:STE} guarantee that
\begin{equation}\label{F-good}
F^\omega:=e^{it\Delta} f^\omega \in X(\R)\cap L_t^\infty L_x^2(\R\times\R^4)\cap L_t^\infty L_x^4(\R\times\R^4)
\end{equation}
almost surely, where $X(\R)$ is as in \eqref{X}.

Now fix $\omega$ such that \eqref{F-good} holds.  Writing $u=F^\omega +v$, we see that $u$ solving \eqref{nls} with $u(0)=f^\omega$ is equivalent with $v$ solving \eqref{fnls} with $v(0)=0$.  By Proposition~\ref{P:energy-bds}, $v$ is uniformly bounded in $H^1_x$ throughout its maximal lifespan.  Thus, by Proposition~\ref{P:v-scatter}, $v$ is a global solution to \eqref{fnls} and scatters in $H^1_x$, which gives Theorem~\ref{T}.\end{proof}


\end{document}